%% file: graphon-distances-arxiv.tex
\documentclass[11pt]{article}
\usepackage[utf8]{inputenc} % use a civilized charset
\usepackage[T1]{fontenc}
\usepackage{geometry}
\usepackage[UKenglish]{babel} % hon hon hon
\usepackage{hhline}% double underline
\usepackage{amsmath} % Math symboles
\usepackage{amsthm}
\usepackage{amsfonts}
\usepackage{amssymb}
\usepackage{graphicx}

\usepackage[bookmarksopen=true, hidelinks]{hyperref}
\usepackage{bookmark}
\usepackage[toc,page]{appendix}
\usepackage{url}
\usepackage{color}
\usepackage{shadow}
\usepackage{fancybox}
\usepackage[babel=true,autostyle=false, style=english]{csquotes}
\MakeOuterQuote{"} % automatically render quotes correctly

\usepackage{bm}
\usepackage{pxfonts}
\usepackage{parskip} % remove paragraph indent and add empty line between paragraphs
\usepackage{enumitem}

\input{commands.tex}

\input{thm.tex}

\date{\today}

\title{The Shortest-Path distance on graphons}

\author{
\small
Cédric Simal\textsuperscript{1},
Julien Petit\textsuperscript{2},
Timoteo Carletti\textsuperscript{1}
}

\begin{document}
\maketitle
{
\small
% Enter affiliation(s) here
1. Department of Mathematics and naXys, University of Namur, Belgium\\
2. Royal Military Academy, Brussels, Belgium
}

\begin{abstract}
    We define an analogue of the shortest-path distance for graphons. The proposed method is rooted on the extension to graphons of Varadhan's formula, a result that links the solution of the heat equation on a Riemannian manifold to its geodesic distance. The resulting metric is integer-valued, and for step graphons obtained from finite graphs it is essentially equivalent to the usual shortest-path distance. We further draw a link between the Varadhan distance and the communicability distance, that contains information from all paths, not just shortest-paths, and thus provides a finer distance on graphons along with a natural isometric embedding into a Hilbert space.
\end{abstract}

\input{sections/intro.tex}

\input{sections/graphons.tex}

\input{sections/varadhan.tex}

\input{sections/graphon-varadhan.tex}

\input{sections/communicability.tex}

\input{sections/conclusion.tex}

\newpage
\input{appendix.tex}

\bibliography{sources}

\end{document}

%% file: commands.tex
\def\N{\mathbb{N}}
\def\R{\mathbb{R}}

\def\log{\operatorname{log}}
\def\supp{\operatorname{supp}}
\def\ker{\operatorname{ker}}
\def\dim{\operatorname{dim}}
\def\ess{\operatorname{ess}}

\def\AA{\mathbb{A}}
\def\WW{\mathbb{W}}
\def\II{\bm{1}}
\def\calP{\mathcal{P}}
\def\calW{\mathcal{W}}
\def\calL{\mathcal{L}}
\def\frakM{\mathfrak{M}_0}

\newcommand{\sans}[1]{\mathsf{#1}}

\def\adj{\sans{adj}}
\def\step{\sans{step}}
\def\lift{\sans{lift}}
\def\mat{\sans{mat}}
\def\coarsen{\sans{coarsen}}

%% file: thm.tex
\newtheorem{theorem}{Theorem}[section] 
\newtheorem{corollary}{Corollary}[section]
\newtheorem{lemma}{Lemma}[section]
\newtheorem{proposition}{Proposition}[section]

\theoremstyle{definition}
\newtheorem{definition}{Definition}[section]
\newtheorem{example}{Example}

%% file: sections/intro.tex
% LTeX: language=en-GB
\section{Introduction}

Networks have long been a versatile tool for modelling a wide variety of phenomena, from physical systems with coupled interactions, to social dynamics and the brain \cite{BarabasiNetworkScience2016,NewmanNetworks2018}. With advances in computing technology, the size of empirical networks that can effectively be recorded and processed has been steadily increasing. However, some of the largest known empirical networks, like the World Wide Web or the human brain remain challenging to deal with. 

An effective approach when dealing with such large discrete objects is to replace them with a continuous abstraction. In the context of networks, such abstractions are called \textit{graph limits}, and generally aim to describe "limit objects" for sequences of graphs of increasing size \cite{lovaszLimitsDenseGraph2006b,benjaminiRecurrenceDistributionalLimits2001,backhauszActionConvergenceOperators2022,borgsLimitsSparseConfiguration2021,borgsLpTheorySparse2019,budelRandomHyperbolicGraphs2024a}. Among these, \textit{graphons} \cite{lovaszGraphon2012} describe the limits of sequences of dense graphs, and have been used in statistics as a general non-parametric model of random graphs \cite{diaconisGraphLimitsExchangeable2008,borgsGraphonsNonparametricMethod2017a}, and in network dynamics to rigorously describe the mean-field limits of dynamical systems on networks \cite{Medvedev2013TheNH,Petit2019RandomWO,vizueteGraphonBasedSensitivityAnalysis2020,bramburgerPatternFormationRandom2023,nagpalSynchronizationRandomNetworks2024} and mean-field games \cite{Parise2018GraphonGA,cainesGraphonMeanField2021a}.

Many centrality measures and network metrics do not yet have a well-defined graphon analogue, with some notable exceptions \cite{AvellaMedinaCentralities2020,GaoFixedPointCentrality2022,klimmModularityMaximizationGraphons2022}. In particular, there still lacks an equivalent to the shortest-path distance, which is relevant to extend distance-based network metrics like closeness centrality \cite{BavelasCloseness1950,Sabidussi_1966} and graph curvature \cite{ollivierRicciCurvatureMarkov2009, linRicciCurvatureGraphs2011,vanderhoornOllivierRicciCurvatureConvergence2021} to graphons. The main obstacle to doing so is that the notion of a path on a graphon is not well-defined. More broadly, defining a distance on graphons enables their study as geometric objects \cite{lovaszRegularityPartitionsTopology2010b,lovaszSzemerediLemmaAnalyst2007a,lovaszAutomorphismGroupGraphon2014b}.

In this work, we propose a definition of shortest-path distance for graphons that leverages the \textit{Varadhan formula} \cite{varadhanBehaviorFundamentalSolution1967,varadhanDiffusionProcessesSmall1967} -- a classical result from differential geometry linking the solution of the heat equation on a Riemannian manifold with its geodesic distance function, and that has been adapted to graphs \cite{kellerNoteShorttimeBehavior2016,steinerbergerVaradhanAsymptoticsHeat2019}. The distance we obtain is integer-valued, and for a step graphon obtained from a finite graph, the distance between two points is completely described by the usual shortest-path distance on the graph.

The Varadhan distance is rather coarse, however, and to obtain a finer notion of distance on graphons, we make a link between the proposed extension of Varadhan's formula and the \textit{communicability distance} \cite{estradaCommunicabilityDistanceGraphs2012}, that contains information from all paths, not just shortest paths, and can be straightforwardly extended to the graphon case while preserving many of its properties. In particular, it comes with a natural isometric embedding of the nodes, as we will later show.

The paper is structured as follows. Section \ref{sec:graphons} provides the basic definitions of graphons, while Sections \ref{sec:graphon-metrics} and \ref{sec:varadhan} cover the existing metrics on graphons and the Varadhan formula respectively. Section \ref{sec:varadhan-graphon} defines the shortest-path distance on graphons via Varadhan asymptotics, and Section \ref{sec:communicability} extends the communicability distance to the graphon case.

In the rest of the text, $\frakM$ denotes the set of Lebesgue measurable sets $S \subset [0,1]$ with strictly positive measure $\mu(S) > 0$. The complement of a measurable set $S$ is denoted by $S^c = [0,1] \setminus S$. Unless otherwise specified, '$\supp$' denotes the support of a measurable function and equality of functions denotes equality almost everywhere.

%% file: sections/graphons.tex
% LTeX: language=en-GB
\section{Graphons}\label{sec:graphons}

\subsection{Definitions} \label{sec:definitions}

The goal of this section is to set up the background definitions on graphon theory for the rest of the text. We refer the interested reader to the book of Lovász \cite{lovaszGraphon2012} for a comprehensive treatment of the theory.

A \textit{graphon} is a symmetric measurable function $W : [0,1]^2 \rightarrow [0,1]$ that is interpreted as the analogue of the adjacency matrix of the continuum limit of a sequence of graphs. The use of $[0,1]^2$ as the domain here is not mandatory, and a more general definition of a graphon can be provided with $W: J \times J \rightarrow [0,1]$, where $(J,\mathcal{A},\pi)$ is a standard probability space. When needed, we will refer to the pair $(J,W)$ as a \textit{generalized graphon}. In this work we will restrict ourselves to graphons supported on $[0,1]^2$, but our results can straightforwardly be extended to the general case. We use $\WW$ to denote the space of graphons.

Given a finite graph $G = (V,E)$, with vertices $V = \{1,\dots,n\}$, we can construct a graphon by viewing its adjacency matrix as a "pixel matrix" on the unit square. More precisely, let $A \in \R^{n\times n}$ be the adjacency matrix of $G$, that is, $A_{ij} = 1$ if the edge $\{i,j\}$ is in $E$, and $A_{ij} = 0$ otherwise. We denote the graphon associated with $G$ as $W_G = \lift_n (A) $, with
\begin{align*}
 \lift_n : \AA_n &\rightarrow \WW \\
 A & \mapsto \sum_{i,j=1}^n A_{ij} \II_{I^{(n)}_i \times I^{(n)}_j},
\end{align*}
where $\AA_n = \{A\in [0,1]^{n\times n}\,|\, A^T=A\}$ is the set of symmetric matrices with entries in $[0,1]$, $\II_S$ denotes the characteristic function of the set $S$ and the sets $I^{(n)}_i = [\frac{i-1}{n},\frac{i}{n})$, $i=1,\dots,n-1$, $I^{(n)}_n = [\frac{n-1}{n},1]$ form a uniform partition of the unit interval.

This construction can be generalized to allow for arbitrary partitions. Let $\calP_n$ denote the set of partitions of the unit interval into $n$ measurable sets with non-zero measure (we shall refer to those as \textit{measurable partitions} in the rest of the text). We define the map
\begin{align*}
    \step_n : \calP_n \times \AA_n &\rightarrow \WW\\
    (\{P_1,\dots,P_n\},A) &\mapsto \sum_{i,j=1}^n A_{ij} \II_{P_i \times P_j}.
\end{align*}
A graphon constructed in this way will be called a \textit{step graphon}, or a \textit{$P$-step graphon} when referring to a specific partition $P$.

We may go the opposite way and obtain a weighted adjacency matrix from a graphon and measurable partition via the map
\begin{align*}
    \mat_n : \calP_n \times \WW &\rightarrow \AA_n\\
    (\{P_1,\dots,P_n\}, W) &\mapsto w = [w_{ij}],
\end{align*}
where $w_{ij} = \frac{1}{\mu(P_i) \mu(P_j)} \int_{P_i \times P_j} W(x,y) dx dy$. When performing these operations with a given partition $P$, we will use the notations $\step_P$ and $\mat_P$. Composing these two operators, we obtain the map $\coarsen_P : \WW \rightarrow \WW = \step_P \circ \mat_P$, which produces the best approximation of $W$ by a $P$-step function with respect to the $L^2$-norm. Equivalently, $\coarsen_P$ is the orthogonal projection operator on the space of $P$-step graphons.

\subsection{Graphon isomorphism and graphon convergence}

As bounded measurable functions, graphons can naturally be compared with any $L^p$ norm, but these are agnostic to the "graph" structure of graphons. For example, two isomorphic graphs may lift to different graphons.

The space of graphons can be equipped with an appropriate distance called the \textit{cut-metric}. Given $W \in \WW$, the \textit{cut-norm} of $W$ is defined as
$$ \|W\|_\square = \sup_{S,T \subset [0,1]} \left|\int_{S\times T} W(x,y) dx dy\right|, $$
where the supremum is taken over measurable sets. The cut-metric is then defined as
$$ \delta_\square (W',W) = \inf_{\varphi} \|W'- W^\varphi \|_\square,$$
where the infimum is taken over measure preserving bijections $\varphi :[0,1] \rightarrow [0,1]$ and $W^\varphi (x,y) = W(\varphi(x),\varphi(y))$. 

The reason for this subtlety is that the map from a finite graph to its adjacency matrix is not unique, and depends on the labelling of its vertices. Given two adjacency matrices of isomorphic graphs, we can turn one into the other by applying a permutation to its rows and columns, which translates to applying a measure preserving bijection in the continuous case.

Two graphons $U,W$ are called \textit{isomorphic} if $U = W^\varphi$ almost everywhere, for some measure preserving bijection $\varphi$. It is clear that if $U$ and $W$ are isomorphic, then $\delta_\square (U,W) = 0$, but the converse is not true. Two graphons $U, W \in \WW$ are called \textit{weakly isomorphic} if $\delta_\square (U,W) = 0$ and after identifying weakly isomorphic graphons, the space of graphons is a complete metric space that is furthermore compact and contains the set of (lifted) finite graphs as a dense subset. 

The metric topology given by the cut-distance therefore induces a notion of convergence of sequences of finite graphs. Intuitively, the cut-metric captures the idea that two graphs are close if the density of edges between arbitrary pairs of subsets of vertices are close for both graphs. Visually, (see Figure \ref{fig:graphon-convergence}), this means that the black and white "pixel matrices" described in Section \ref{sec:definitions} converge to a smooth greyscale picture.

Weak isomorphism and convergence can be extended to generalized graphons, and it can be shown that any generalized graphon is weakly isomorphic to a graphon on $[0,1]$, which justifies focusing on that particular case.

\begin{figure}
    \centering
    \includegraphics[alt={Graphical representation of step graphons approximating a limiting continuous graphon.},width=0.9\textwidth]{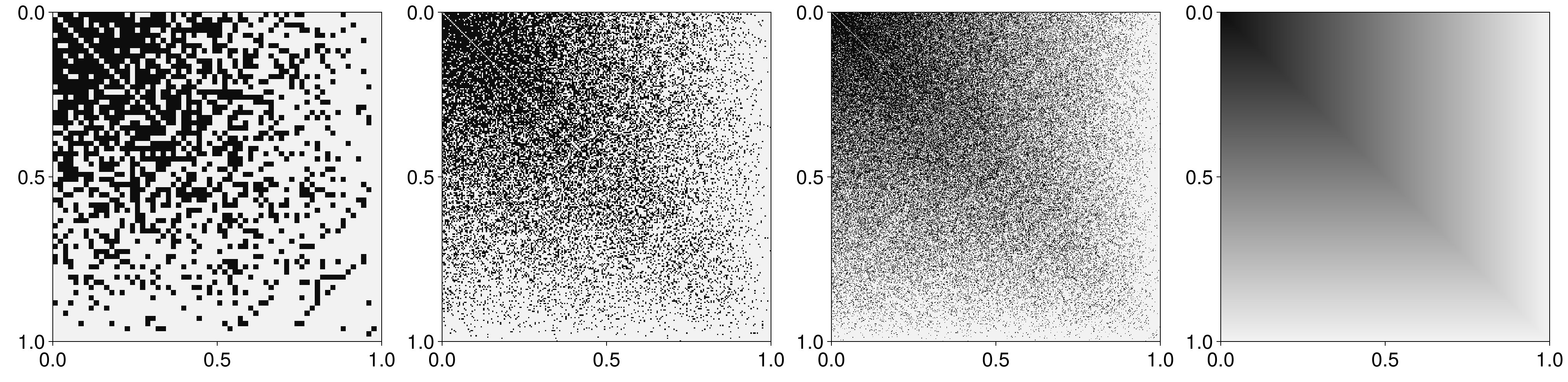}
    \caption{A sequence of finite graphs converging in the cut-metric to the graphon $W(x,y) = 1-\max(x,y)$. White pixels represent a value of zero, while black pixels represent a value of one.\\
    Alt text: Graphical representation of step graphons approximating a limiting continuous graphon.
    }
    \label{fig:graphon-convergence}
\end{figure}

\subsection{Graphon operators}

Graphons can be naturally viewed as kernels of integral operators on spaces of measurable functions $[0,1] \rightarrow \R$. Let us in particular consider the space $L^2 [0,1]$ of square integrable functions and some graphon $W \in \WW$. The \textit{adjacency operator} $\calW = \adj (W)$ is the linear operator defined by
$$ (\calW f)(x) = \int_0^1 W(x,y) f(y) dy, $$
for $f \in L^2 [0,1]$ and $x \in [0,1]$. Integer powers of the adjacency operator are also integral operators, with the integral kernel of $\calW^m$ being given by the \textit{composition power}
$$ W^{\circ m} (x,y) = \int_{[0,1]^{m-1}} W(x,z_1) \prod_{i=1}^{m-2} W(z_i,z_{i+1}) W(z_{m-1},y) dz_1 \dots dz_{m-2}. $$

If a graphon $W$ is to be interpreted as an adjacency matrix, we can define its \textit{degree function} $k$ as 
$$ k(x) = \int_0^1 W(x,y) dy = (\calW \II_{[0,1]})(x). $$
It can eventually be used to define the \textit{combinatorial Laplacian operator} $\calL$ as
$$ (\calL f)(x) = k(x) f(x) - \int_0^1 W(x,y) f(y) dy, $$
or in operator notation, as $\calL = T_k - \calW$, where $T_k$ denotes the multiplication operator by $k$.

Both the adjacency operator and the Laplacian are bounded and self-adjoint on $L^2 [0,1]$, and in addition, the adjacency operator is compact. The Laplacian operator, however, is not compact.

The following proposition shows how to interpret graphon isomorphism from the operator point of view. Despite its simplicity, we could not find it in the existing literature, and thus we prove it for the sake of completeness.
\begin{proposition} \label{prop:isomorphic-graphon-unitary}
    If two graphons are isomorphic, then their adjacency operators are unitarily equivalent.
\end{proposition}
\begin{proof}
    Let $W \in \WW$ and $\varphi: [0,1]\rightarrow [0,1]$ a measure preserving bijection. We construct a unitary operator $U^\varphi$ on $L^2 [0,1]$ induced by $\varphi$ as
    $$ (U^\varphi f)(x) = f(\varphi(x)). $$
    The operator $U^\varphi$ is easily proved to be linear, and letting $(U^\varphi)^*$ denote the adjoint of $U^\varphi$, we have for all $f,g\in L^2 [0,1]$,
    \begin{align*}
        \langle (U^\varphi)^* f, g\rangle = \langle f, U^\varphi g\rangle &= \int_0^1 f(x) g(\varphi(x)) dx = \int_0^1 f(\varphi^{-1}(x)) g(x) dx,
    \end{align*}
    so $((U^\varphi)^* f)(x) = f(\varphi^{-1}(x))$, and it is clear that the $(U^\varphi)^* = (U^\varphi)^{-1} = U^{\varphi^{-1}}$, so $U^\varphi$ is unitary.

    Now, for $x\in [0,1]$, we compute
    \begin{align*}
        (U^\varphi \calW U^{\varphi^{-1}} f)(x) & = \int_0^1 W(\varphi(x), y) f(\varphi^{-1}(y)) dy = \int_0^1 W(\varphi(x), \varphi(y)) f(y) dy,
    \end{align*}
    so $U^\varphi \calW U^{\varphi^{-1}}$ is equal to the adjacency operator of $W^\varphi$, which is isomorphic to $W$. Since any graphon isomorphic to $W$ can be expressed in this form, we conclude the proof.
\end{proof}

\begin{corollary} \label{coro:powers-isomorphism}
    For all $n\ge 1$,
    $$ (W^\varphi)^{\circ n} = (W^{\circ n})^\varphi $$
\end{corollary}
\begin{proof}
    Let $\calW_\varphi = \adj(W^\varphi)$. The decomposition obtained in the previous proposition implies, for all $n\ge 1$, that
    \begin{equation} \label{eq:adjacency-power-isomorphism}
        \calW^n_\varphi = U^\varphi \calW^n U^{\varphi^{-1}}.
    \end{equation}
    Each of these operators are kernel integral operators with kernels $(W^\varphi)^{\circ n}$ and $(W^{\circ n})^{\varphi}$ respectively. From the equality \eqref{eq:adjacency-power-isomorphism}, we conclude that they must be equal almost everywhere.
\end{proof}

\subsection{Metrics on graphons}\label{sec:graphon-metrics}

The main instances of metrics defined on graphons are the \textit{neighbourhood} and \textit{similarity} distances \cite{lovaszRegularityPartitionsTopology2010b,lovaszAutomorphismGroupGraphon2014b}, which we recall here.

Let $W \in \WW$ be a graphon. The neighbourhood distance $r_W : [0,1]^2 \rightarrow \R$ is defined for almost all $x,y\in [0,1]$ as
$$ r_W (x,y) = \| W(x,\cdot) - W(y,\cdot) \|_1. $$
This is not necessary well-defined for all $x,y$, as $W(x,\cdot)$ may fail to be measurable, but only on a set of points of zero measure. There are several steps needed to obtain well-defined metric space.

First, restrict $r_W$ to the domain $I$ of points where it is well-defined, which makes $r_W$ a pre-metric. Second, there may be points $x \neq x'$ such that $W(x,\cdot) = W(x',\cdot)$ almost everywhere. It is possible to "merge" such points to construct a generalized graphon $W': J^2 \rightarrow [0,1]$ weakly isomorphic to $W$ and such that $(J,r_{W'})$ is a complete metric space \cite{lovaszRegularityPartitionsTopology2010b}. Such a graphon is then called a \textit{pure graphon}.

The neighbourhood distance owes its name to the fact that it can be interpreted as measuring the difference between the neighbourhoods of $x$ and $y$. A more interesting metric is the $2$-neighborhood distance, called the \textit{similarity distance} $\overline{r}_W = r_{W^{\circ 2}}$ that can be shown to satisfy $\overline{r}_W \le r_W$. When $(J,W)$ is a pure graphon, $(J,\overline{r}_W)$ is a metric space that is not necessarily complete. As with the neighbourhood distance, it is possible to construct a graphon $(\overline{J},\overline{W})$ on the completion $\overline{J}$ of $J$, weakly isomorphic to $(J,W)$ such that $(\overline{J},\overline{r}_{\overline{W}} )$ is again a complete metric space, with $\overline{r}_{\overline{W}} = \overline{r}_W$ on $J$.

The reward for going through all these steps is the following. The metric space $(\overline{J}, \overline{r}_{\overline{W}})$ is compact, and satisfies many useful properties. In particular, eigenfunctions of the adjacency operator are continuous in this topology \cite{lovaszRegularityPartitionsTopology2010b}.

\begin{example} \label{ex:bipartite}
    The \textit{bipartite graphon} is defined as 
    $$ W(x,y) = \begin{cases}
        1 & \text{ if } (x,y) \in I_1 \times I_2 \cup I_2 \times I_1,\\
        0 & \text{ else},
    \end{cases} $$
    where $I_1 = [0,\frac{1}{2})$ and $I_2=[\frac{1}{2},1]$. This is a step graphon with partition $I^{(2)}=\{I_1,I_2\}$ and matrix
    $$ P = \begin{bmatrix}
        0 & 1\\ 1 & 0
    \end{bmatrix}, $$
    which is the adjacency matrix of the graph with two nodes connected by an edge. The neighbourhood and similarity distances are the same, and $r_W(x,y)=1$ if $x$ and $y$ are in different sets $I_1$, $I_2$ and zero otherwise. By identifying points $x,y$ such that $r_W(x,y)=0$, we obtain a discrete metric space with two points (corresponding to $I_1$, $I_2$) at a distance of one from each other.
\end{example}

\begin{example} \label{ex:circular-band}
    The \textit{circular band graphon} is defined as
    $$ W(x,y) = \begin{cases}
        1 & \text{ if } \Delta(x,y) \le \frac{1}{4},\\
        0 & \text{ if } \Delta(x,y) > \frac{1}{4},
    \end{cases} $$
    where $\Delta(x,y) = \min(|x-y|, 1-|x-y|)$ is the normalized circular distance. This is a pure graphon, with the neighbourhood distance equal to $\Delta$. It can be shown that $W^{\circ 2}(x,y) = \frac{1}{2}-\Delta(x,y)$, so the neighbourhood and similarity distances are equivalent.
\end{example}

%% file: sections/varadhan.tex
% LTeX: language=en-GB
\section{Intrinsic metrics}\label{sec:varadhan}

As seen in the previous section, the existing distances on graphons take some care to define. The goal of the next few sections is to develop an alternative distance on graphons that is easier to work with.

\subsection{Varadhan's formula for manifolds}
Varadhan's formula is a celebrated result from differential geometry linking the fundamental solution of the heat equation on a differential manifold to its metric \cite{varadhanBehaviorFundamentalSolution1967,varadhanDiffusionProcessesSmall1967}. Its starting point is the observation that the solution of the heat equation on $\R^k$,
\begin{equation} \label{eq:heat} 
    \frac{\partial q}{\partial t} = \frac{1}{2} \sum_{i=1}^k \frac{\partial^2 q}{\partial y_i^2},
\end{equation}
with boundary condition $\lim_{t \rightarrow 0^+} q(t) = \delta_x$ (the Dirac delta at some $x \in \R^k$) depends explicitly on the Euclidean distance between $x$ and $y$. More precisely, the solution $q_t (x,y)$ with "source" $x$ is known to be given by
\begin{equation} 
    q_t(x,y) = (2\pi t)^{-k/2} \exp \left(-\frac{1}{2t} \|x-y\|^2 \right), 
\end{equation}
that is, a Gaussian function centred at $x$ whose variance increases with time. The function $q_t (x,y)$ is called the \textit{heat kernel}, and corresponds to the \textit{fundamental solution} (also known as Green's function) of \eqref{eq:heat}, in the sense that for any compactly supported smooth initial condition $f_0$, the solution $f_t$ of the heat equation can be computed as the integral transform (known as the \textit{heat semigroup})
$$ f_t(x) = \int_{\R^k} q_t (x,y) f_0(y) dy. $$
Taking the logarithm of the heat kernel, we find
$$ \log q_t (x,y) = -\frac{k}{2} (\log t + \log 2\pi) - \frac{1}{2t} \|x-y\|^2, $$
so that by multiplying by $2t$ and taking the limit as $t\rightarrow 0^+$, the Euclidean metric $d(x,y) = \|x-y\|$ can be recovered by using the formula
$$ d^2 (x,y) = \lim_{t \rightarrow 0^+} [-2t \log q_t (x,y)], $$
known as \textit{Varadhan's formula}. Varadhan \cite{varadhanBehaviorFundamentalSolution1967} further showed that applying this formula to the solution of
$$ \frac{\partial q}{\partial t} = \frac{1}{2} \sum_{i=1}^k a_{ij} (y) \frac{\partial^2 q}{\partial y_i \partial y_j}, $$
with the same initial condition yields a distance $d$ induced by a Riemannian metric derived from the coefficients $a_{ij}(y)$. Finally, replacing the Dirac delta $\delta_x$ by the characteristic function of some set $B$ yields the distance to the boundary of $B$. Beyond Varadhan's formula, there is a large body of work using the heat equation to study the "geometry" of objects other than manifolds, such as metric measure spaces and fractals \cite{molchanovDIFFUSIONPROCESSESRIEMANNIAN1975,grigoryanAnalysisFractalSpaces2022,grigoryanHeatKernelsMetric2003a}, as well as numerical applications \cite{craneHeatMethodDistance2017,sharpVectorHeatMethod2019,belyaevADMMbasedSchemeDistance2020}.

\subsection{Varadhan's formula for graphs} \label{sec:varadhan-graphs}

Let $G = (V,E)$ be an undirected graph with $n$ nodes. The heat equation on $G$ is given by
\begin{equation} \label{eq:graph-heat}
 \dot{u} = -Lu, 
\end{equation}
where for $t\ge 0$, $u(t) \in \R^n$ and $L=D-A$ is the graph combinatorial Laplacian, with $A$ the adjacency matrix of $G$, and $D$ the diagonal matrix of degrees ($D_{xx} = \sum_{y}A_{xy}$). The solution of \eqref{eq:graph-heat} with initial condition $u_0 \in \R^n$ can be explicitly written as $ u(t) = e^{-Lt} u_0$, where $e^{-Lt}$ is the matrix exponential defined by its Taylor series
\begin{equation} \label{eq:heat-kernel}
 e^{-Lt} = \sum_{k=0}^\infty \frac{(-t)^k}{k!} L^k. 
\end{equation}

The heat kernel $q_t (x,y)$ can be defined as the solution of \eqref{eq:graph-heat} with initial condition $\delta_x$, evaluated at $y$ (for $x,y \in V$), that is
$$ q_t(x,y) = (e^{-Lt})_{xy} = \langle \delta_y, e^{-Lt} \delta_x \rangle, $$
where $\delta_x \in \R^n$ denotes the canonical basis vector with entries $(\delta_x)_y = \delta_{xy}$ (this latter expression being the Kronecker delta).

Previous work studied Varadhan-type asymptotics in the case of graphs \cite{kellerNoteShorttimeBehavior2016,steinerbergerVaradhanAsymptoticsHeat2019}, showing that the short-time behaviour of the heat kernel is given by
$$ q_t (x,y) = c(x,y) t^{d(x,y)} + O(t^{d(x,y)+1}), $$
where $d(x,y)$ is the shortest path distance on the graph between $x$ and $y$, and $c(x,y)$ is a constant equal to the number of such shortest paths. This lets us state the following Varadhan formula for graphs \cite{kellerNoteShorttimeBehavior2016}:
\begin{equation} \label{eq:graph-varadhan} 
    d(x,y) = \lim_{t\rightarrow 0^+} \frac{\log q_t (x,y)}{\log t}.
\end{equation}

The intuition behind this result can be found by looking at the Taylor series of the heat kernel \eqref{eq:heat-kernel} and recalling the well-known fact that for $m \in \N$, $(A^m)_{xy}$ is equal to the number of walks of length $m$ between $x$ and $y$. In particular, $(A^m)_{xy} = 0$ for $m < d(x,y)$, and $d(x,y)$ is the smallest integer $\ell$ such that $(A^\ell)_{xy} \neq 0$, and counts the number of paths of length $\ell$ between $x$ and $y$.

A similar claim holds true for powers of the Laplacian. We show this result in the simple case where $D$ and $A$ commute (meaning vertices in the same connected component have the same degree). This implies that
$$ L^m = (D-A)^m = \sum_{k=0}^m \binom{m}{k} (-1)^{m-k} D^k A^{m-k}. $$
When $m < d(x,y)$, the $xy$ entry of each term $D^k (A^{m-k})$ is zero, hence $(L^m)_{xy} = 0$. When $m = d(x,y)$, all terms in the sum are zero, except for the first one, which is equal to $(A^{d(x,y)})_{xy}$.

In the case where $D$ and $A$ do no commute, we can no longer use the binomial formula as above, and instead the $m$-th power of the Laplacian is expressed as a sum over all possible products of $m$ terms chosen from $\{D,A\}$, e.g. $L^2 = D^2 - DA - AD + A^2$. The same reasoning applies, however. The element $(L^m)_{xy}$ will be non-zero only when $m\ge d(x,y)$.

Varadhan's formula for graphs can be proven for the more general case of graphs with up to countably many vertices and bounded degree by using the following theorem of \cite{kellerNoteShorttimeBehavior2016}, which we will use to define the shortest path distance on graphons.

\begin{theorem}[Corollary 1.4 of \cite{kellerNoteShorttimeBehavior2016}] \label{thm:keller}
    Let $L$ be a self-adjoint positive semidefinite operator on the Hilbert space $(\mathcal{H},\langle \cdot,\cdot\rangle)$. Let $f,g \in D(L^n)$ (the domain of $L^n$) for all $n\in \N$. Let
    $$ d_L (f,g) = \inf \{n \in \N | \langle f,L^n g\rangle \neq 0\}. $$
    The following statements hold.
    \begin{enumerate}
        \item If $d_L(f,g) < \infty$, then
        $$ \lim_{t\rightarrow 0^+} \frac{\log |\langle f, e^{-Lt}g \rangle|}{\log t} = d_L (f,g). $$
        \item If $d_L (f,g) = \infty$, there exists for all $n\in\N$ a constant $C_n(f,g) >0$ such that
        $$ |\langle f, e^{-Lt}g\rangle | \le C_n (f,g) t^{n+1}. $$
    \end{enumerate}
\end{theorem}

\subsection{Families of Varadhan formulas on graphs} \label{sec:other-odes}

In the proof of the graph version of Varadhan's formula, we essentially relied on the property that $(A^m)_{xy}$ is the number of walks of length $m$ between vertices $x$ and $y$, which means that the first non-zero term in the Taylor expansion of the heat kernel is of order $d(x,y)$. The heat equation is not the only equation with this property, and we could just as well have used the simpler equation
\begin{equation} \label{eq:adjacency} 
    \dot{u} = Au,
\end{equation}
whose solution is given by $u(t) = e^{At}u_0$, where the exponential $e^{At}$ has the same dominant term in the series expansion. Numerically, this equation is less stable than the heat equation, but is easier to work with analytically, as we no longer have to deal with the degree matrix $D$ not commuting with $A$.

Yet another equation that will satisfy the Varadhan formula is
$$ \dot{u} =  (I - tA)^{-1} A (I-tA)^{-1} u, $$
whose solution is given by $u(t) = (I-tA)^{-1} u_0$, for $t \ll 1$. This illustrates that, in fact, equation \eqref{eq:adjacency} can be replaced with a generic matrix ODE, provided its solution is analytic in a neighbourhood of the origin, namely
$$ \dot{u}(t) = g(A, t) u(t), $$
such that $u(t) = f(At) u_0$, with $f$ analytic near zero. A more general form of Varadhan's formula for graphs is then given by
$$ \lim_{t\rightarrow 0^+} \frac{\log |f(At)_{xy}|}{\log t} = d(x,y). $$

This can be generalized further to a large class of Taylor series and matrices associated to the graph. More precisely, 

\begin{enumerate}
    \item The adjacency matrix can be replaced by any non-negative matrix $M \in \R_+^{n\times n}$ such that $M_{xy} = 0$ if and only if $A_{xy}=0$. The sign of $M$ can be flipped.
    \item An arbitrary diagonal matrix can be added to $M$.
    \item The function $f$ can be any analytic function near the origin such that all its Taylor coefficients are non-zero.
\end{enumerate}

We provide a formal statement and proof of this result the supplementary materials.

%% file: sections/graphon-varadhan.tex
% LTeX: language=eb-GB
\section{Varadhan's formula for graphons} \label{sec:varadhan-graphon}

In the following sections, we will gradually build up to a pointwise distance on the unit interval induced by a graphon $W$. We will first need to determine conditions to obtain a well-defined distance. For a finite graph, the shortest path-distance is only well-defined if the graph is connected. We will study the equivalent notion for graphons in Section \ref{sec:graphon-connectedness}. 

Once this point has been achieved, we will use Theorem \ref{thm:keller} in Section \ref{sec:graphon-heat} to obtain a positive symmetric function $\delta_W$ on pairs of subsets in $\frakM$. This function is not a metric on $\frakM$, except when restricting its domain to members of a measurable partition.  Finally, in Section \ref{sec:varadhan-pointwise}, we will arrive at a metric $\overline{\delta}_W$ on $[0,1]$ by taking the limit of $\delta_W$ on nested neighbourhoods of the points. The main result of that section will provide a characterization of $\overline{\delta}_W$ in terms of the supports of composition powers of $W$.

\subsection{Connectedness of a graphon} \label{sec:graphon-connectedness}

Before moving on with the definition and properties of the Varadhan distance on graph\-ons, we present the following definitions, that aim at extending the notion of graph connectedness.

\begin{definition}[Connected graphon \cite{jansonConnectednessGraphLimits2008}] \label{def:connected-graphon}
A graphon $W \in \WW$ is said to be \textit{connected} if for any measurable set $U \in \mathfrak{M}_0$ such that $\mu(U) \in (0,1)$,
$$ \langle \II_U, \calW \II_{U^c} \rangle = \int_{U \times U^c} W(x,y) dx dy > 0. $$
\end{definition}

In addition to this definition, we introduce the following one with Theorem \ref{thm:keller} in mind. The main result of this section will be the equivalence of these two definitions.

\begin{definition}[Finitely connected graphon] \label{def:finitely-connected-graphon}
    A graphon $W \in \WW$ is \textit{finitely connected} if for any $U,V \in \mathfrak{M}_0$, there exists $m \in \N$, such that
    $$ \langle \II_U, \mathcal{W}^m \II_V \rangle > 0.  $$

    If moreover, there exists some $M \in \N$ such that for any $U$, $V$, we have $m \in \N$ as above with $m \le M$, then $W$ is said to be \textit{uniformly finitely connected} and the smallest such $M$ is called the \textit{diameter} of $W$.
\end{definition}

A graphon $W$ is said to be \textit{trivially connected} if it is positive almost everywhere. Trivially connected graphons are both connected and finitely connected. 

\begin{lemma} \label{lemma:not-connected}
    If $W$ is not connected, then $W^{\circ m}$ is not connected for all $m \ge 1$.
\end{lemma}
\begin{proof}
    Let $\calW$ be the adjacency operator of $W$ and $S \in \frakM$ such that $\mu(S) \in (0,1)$ and $\langle \II_{S^c}, \calW \II_S \rangle = 0$. We use this as the base case ($m=1$) to proceed by induction.

    Let $m \ge 1$ be such that 
    \begin{equation} \label{eq:induction}
    \langle \II_{S^c}, \calW^m \II_S \rangle = 0.
    \end{equation}
    Observe that $\II_{S^c}$ and $\calW^m \II_S$ are non-negative functions, so Equation \eqref{eq:induction} implies that $\calW^m \II_S$ is zero almost everywhere on $S^c$, hence $T = \supp (\calW^m \II_S)$ is such that $ \mu(T \setminus S)=0$. Furthermore, $\calW^m \II_S (x) \le 1, \forall x \in [0,1]$, so $(\calW^m \II_S (x)) \le \II_T (x)$. Then,
    \begin{align*}
        \langle \II_{S^c}, \calW^{m+1} \II_S \rangle &\le \langle \II_{S^c}, \calW \II_T \rangle \le \langle \II_{S^c}, \calW \II_{S} \rangle =0. 
    \end{align*}
\end{proof}

\begin{lemma} \label{lemma:neighborhood}
    Let $f \in L^2[0,1]$ be a non-negative function and $W$ a graphon. Then $\calW f$ is a non-negative function and
    $$ \supp (\calW f) = \supp (\calW \II_{\supp(f)}). $$
\end{lemma}
\begin{proof}
    The non-negativity of $\calW f$ comes from the non-negativity of both $W$ and $f$. If $f=0$, the result is immediate, so suppose that $f\neq 0$, and let $S = \supp (f)$ and $x\in [0,1]$. Suppose that $(\calW f)(x) = 0$, then developing the integral and using the definition of $S$, we have
    $$ \int_S W(x,y) f(y) dy = 0. $$
    Since $W$ and $f$ are non-negative, and $f$ is non-zero on $S$, this implies that $W(x,\cdot) = 0$ almost everywhere on $S$, hence $(\calW \II_U)(x) = 0$. The same argument can be run backwards, so that
    $$ (\calW f)(x) = 0 \Leftrightarrow (\calW \II_S)(x) = 0, $$
    so $\supp (\calW f) = \supp (\calW \II_S)$.
\end{proof}

In similar fashion to the proof of Lemma 5.1 of \cite{jansonConnectednessGraphLimits2008}, we define the following "neighborhood operator".

\begin{definition}[$W$-neighbourhood]
    Let $W$ be a graphon and $U \in \frakM$. The $W$-neighbourhood of $U$ is defined as
    $$ \mathcal{N}_W (U) = \supp (\calW \II_U). $$ 
\end{definition}
The map $\mathcal{N}_W$ sends measurable sets to measurable sets, and we denote applying it $m$ times by $\mathcal{N}_W^m$.

\begin{lemma} \label{lemma:neighborhood-powers}
    For all $m \ge 1$,
    $$ \mathcal{N}_W^m (U) = \supp (\calW^m \II_U) $$
\end{lemma}
\begin{proof}
    The case $m=1$ is true by definition. Given $m\ge 1$ such that 
    $$ \mathcal{N}_W^m (U) = \supp (\calW^m \II_U), $$
    we have directly by applying Lemma \ref{lemma:neighborhood} that
    \begin{align*}
        \supp(\calW^{m+1} f) &= \supp (\calW (\calW^m \II_U)) = \supp (\calW \II_{\mathcal{N}_W^m (U)}) = \mathcal{N}_W^{m+1} (U).
    \end{align*}
\end{proof}

We now show that being finitely connected is equivalent to being connected.

\begin{theorem}
    Let $W$ be a graphon and $\calL$ its Laplacian operator. Then the following are equivalent.
    \begin{enumerate}[label=(\alph*)]
        \item $W$ is connected,
        \item $W$ is finitely connected,
        \item $\dim(\ker(\calL)) = 1$.
    \end{enumerate}
\end{theorem}
\begin{proof}
    We start by proving that $(a) \Leftrightarrow (b)$.
    \begin{itemize}
        \item $(b) \Rightarrow (a)$ If $W$ is not connected, lemma \ref{lemma:not-connected} implies that $W$ cannot be finitely connected.
        \item $(a) \Rightarrow (b)$ Suppose that $W$ is not finitely connected, and let $U,V \in \frakM$ such that $\langle \II_U, \calW^m \II_V \rangle = 0$ for all $m \in \N$. Since $\calW$ is self-adjoint, this is equivalent to 
        $$\langle \calW^{m} \II_U, \calW^{k} \II_V \rangle = 0,$$ 
        for all $m,k\in\N$. 
        
        We define sequences of sets $(U_n)_{n\in\N}$, $(V_n)_{n\in\N}$ as $U_n = \mathcal{N}_W^n (U)$ (respectively for $V_n$), and consider the "maximal reachable sets" $U_\infty = \bigcup_{n\in \N} U_n$ and $V_\infty = \bigcup_{n\in\N} V_n$. 
        
        Observe that by Lemma \ref{lemma:neighborhood}, we have, for all $m, k \in \N$,
        $$ \langle \II_{U_m}, \II_{V_k}\rangle= 0. $$
        We wish to show two things. First, that $U_\infty$ and $V_\infty$ are closed under taking neighborhoods, i.e.
        $$ \mathcal{N}_W(U_\infty) \subset U_\infty, $$
        (resp. for $V_\infty$), and second, that they are disjoint up to a null set, i.e.
        $$ \langle \II_{U_\infty}, \II_{V_\infty} \rangle = 0. $$

        For the first statement, we may directly apply Lemma 5.1 (iv) from \cite{jansonConnectednessGraphLimits2008}, which states that the neighborhood operator commutes with taking unions.
        \begin{align*}
            \mathcal{N}_W (\cup_{m\in \N} U_m) &= \bigcup_{m\in \N} \mathcal{N}_W (U_m) = \bigcup_{m\in \N} U_{m+1}\subset U_\infty.
        \end{align*}
        The same argument yields $\mathcal{N}_W (V_\infty) \subset V_\infty$.

        For the second statement, we first "orthogonalize" the sets $U_m$ (resp. $V_m$), by defining $\overline{U}_0 = U_0$, $\overline{U}_{m+1} = U_{m+1} \setminus (\cup_{k=0}^m U_k), \forall m \in \N$ (the construction of the sets $\overline{V}_m$ is identical).

        The sets $\overline{U}_m$ (resp. $\overline{V}_m$) are mutually disjoint by construction, and since $\overline{U}_m \subset U_m$ ($\overline{V}_m \subset V_m$) for all $m\in \N$, we have for all $m, k \in \N$,
        \begin{align*}
        \langle \II_{\overline{U}_m}, \II_{\overline{V}_k} \rangle &\le \langle \II_{U_m}, \II_{V_k} \rangle = 0.
        \end{align*}

        Moreover, for any $f \in L^2 [0,1]$, we have
        \begin{align*}
            \langle f, \II_{U_\infty} \rangle &= \sum_{m=0} \langle f, \II_{\overline{U}_m} \rangle,
        \end{align*}
        so $\II_{U_\infty} = \sum_{m=0}^\infty \II_{\overline{U}_m}$ and $\II_{V_\infty} = \sum_{m=0}^\infty \II_{\overline{V}_m}$, in the sense that both those series converge in the $L^2$ norm.

        Finally, we may write
        \begin{align*}
            \langle \II_{U_\infty}, \II_{V_\infty} \rangle &= \sum_{k,m \in \N} \langle \II_{\overline{U}_m}, \II_{\overline{V}_k} \rangle = 0.
        \end{align*}
        And because $U_\infty$ and $V_\infty$ are closed under taking neighborhoods, we have
        $$ \langle \II_{U^c_\infty}, \calW \II_{U_\infty} \rangle = 0, $$
        so taking $U_\infty$ suffices to show that $W$ is not connected, as $V \subset V_\infty \subset U_\infty^c$, so $\mu(U_\infty^c) > 0$.
        \item $(a)\Rightarrow (c)$ We adapt the proof from \cite{Petit2019RandomWO}. Observe that $\II_{[0,1]}$ is always in the kernel of $\calL$, so without loss of generality, we may search for $f \in \ker(\calL)$ such that $\langle f,\II_{[0,1]} \rangle = \int_0^1 f(y) dy = 0$. The following sets therefore both have positive measure
        \begin{align*}
            S &= \{x \in [0,1]\,|\, f(x) \ge 0\},\\
            S^c &= \{x \in [0,1]\,|\, f(x) < 0\}.
        \end{align*}
        If $f \in \ker (\calL)$, i.e. $\calL f = 0$, we must have
        $$ \int_{S\times S^c} W^2(x,y) (f(x)-f(y))^2 dx dy \le \int_{[0,1]^2} W^2(x,y) (f(x)-f(y))^2 dx dy = 0,
        $$
        so the first integral must be zero. Since $W$ is connected, there exists a positive measured subset $U \times V \subset S \times S^c$ on which $W>0$. We must therefore have $f(x)-f(y)=0$ for almost every $(x,y) \in U\times V$, but this is a contradiction since, by construction, $f(x) \ge 0 > f(y)$.
        \item $(c)\Rightarrow (a)$ Suppose that $W$ is not connected and let $S \in \frakM$ be such that $\mu(S) \in (0,1)$ and $\langle \II_{S^c}, \calW \II_S \rangle = 0$. Then for almost all $x \in S$, we have $k(x) = \int_{S} W(x,y) dy$, and it is easy to check that $\II_S$ satisfies $$ \calL \II_S = T_k \II_S - \calW \II_S = 0. $$ The same is true for $\II_{S^c}$ and so $\dim(\ker(\calL)) > 1$.
    \end{itemize}
\end{proof}

Let us remark that definitions \ref{def:connected-graphon} and \ref{def:finitely-connected-graphon} and the proof of their equivalence could be extended to arbitrary self-adjoint operators on $L^2 [0,1]$. This is therefore a result that could be extended to other types of graph limits \cite{backhauszActionConvergenceOperators2022}.

The next two propositions show the relation between graph connectedness and graphon connectedness. Connected graphs correspond to connected step graphons and vice versa.

\begin{proposition}
    Let $A \in \AA_n$ be the adjacency matrix of a connected graph $G$. Then the graphon $W = \lift_n (A)$ is finitely connected.
\end{proposition}
\begin{proof}
    Since $G$ is connected, for any vertices $i,j \in V(G)$, there exists some integer $m = m(i,j)$ such that $(A^m)_{ij} > 0$. For the graphon $W = \lift_n (A)$, by using Corollary B.1 (proved in the supplementary materials), we have that $W^{\circ m} = \lift_n (\frac{1}{n^{m-1}} A^m)$ is non-zero on $I^{(n)}_i \times I^{(n)}_j$ where $I^{(n)}_i, I^{(n)}_j \in I^{(n)}$ are the partition sets associated to vertices $i$ and $j$.

    Now, for arbitrary $U, V \in \frakM$, there must be some $I^{(n)}_i, I^{(n)}_j \in I^{(n)}$ such that $\mu (U \cap I^{(n)}_i) \neq 0$ and $\mu(V \cap I^{(n)}_j) \neq 0$. Then
    \begin{align*}
        \langle \II_V, \calW^m \II_U \rangle &\ge
        \langle \II_{V \cap I^{(n)}_j}, \calW^m \II_{U \cap I^{(n)}_j} \rangle > 0.
    \end{align*}
\end{proof}

\begin{proposition} \label{prop:graphon-coarsen-connected}
    Let $W$ be a connected graphon and $P\in \calP_n$ a measurable partition. Then $W_P = \coarsen_P (W)$ is connected.
\end{proposition}
\begin{proof}
    Let $S \in \frakM$ such that $S^c \in \frakM$. Since $W$ is connected, there must exist a measurable subset $U \times V \subset S \times S^c$ such that $W \neq 0$ almost everywhere on $U \times V$, and furthermore there must exist some $P_i, P_j \in P$ such that $\mu(U \cap P_i) \neq 0$ and $\mu(V \cap P_j) \neq 0$. Let $\calW_P = \adj (W_P)$. Observe that 
    \begin{align*}
        \int_{P_i\times P_j} W_P (s,t) ds dt &= \int_{P_i \times P_j} \left[\frac{1}{\mu(P_1)\mu(P_2)}\int_{P_i\times P_j} W(x,y) dx dy \right] ds dt\\
        &= \int_{P_i\times P_j} W(x,y) dx dy \ge \int_{U \cap P_i \times V \cap P_j} W(x,y) dx dy > 0,
    \end{align*}
    so $W_P$ is nonzero on $P_i \times P_j$, and we then have
    \begin{align*}
        \langle \II_S, \calW_P \II_{S^c} \rangle &\ge \langle \II_U , \calW_P \II_V \rangle \ge \langle \II_{U \cap P_i}, \calW_P \II_{V \cap P_j} \rangle > 0,
    \end{align*}
    so $W_P$ is connected.
\end{proof}

\begin{example}
    The Erd\"os-Rényi graphon is trivially connected and has diameter equal to one.
\end{example}

\begin{example}
    The bipartite graphon is obtained from a connected graph. It is therefore connected and has diameter equal to two. Note that this is larger than the diameter of the graph with two vertices connected by an edge.
\end{example}

\begin{example}
    The circular band graphon defined in example \ref{ex:circular-band} satisfies $W^{\circ 2}(x,y) = \frac{1}{2}-\Delta(x,y)$, which is positive almost everywhere, so that $\langle \II_U, \calW^2 \II_V \rangle >0$ for all $U,V\in\frakM$, so $W$ is connected with diameter equal to two.
\end{example}

\begin{example}
    Consider the infinite path graph with vertices indexed by $\N$ and edge set $\{\{n,n+1\} \,|\, \forall n\in \N\}$. We can construct a graphon from this graph (or any countable graph) by associating vertex $n$ with the interval $[1-\frac{1}{2^n},1-\frac{1}{2^{n+1}})$. This is a piecewise constant graphon (but not a step graphon, because there are infinitely many steps) that is connected, but not uniformly finitely connected.
\end{example}

\subsection{Heat kernels on graphons} \label{sec:graphon-heat}

Let $W \in \WW$ be a graphon. The heat equation on $W$ is given by the following abstract differential equation \cite{Petit2019RandomWO},
$$ \dot{u} = - \mathcal{L} u, $$
with initial condition $u_0 \in L^2 [0,1]$. Because $\mathcal{L}$ is a bounded linear operator, a classical solution exists \cite{CurtainZwart2020}, and is explicitly obtained by using the exponential operator,
$$ u(t) = e^{-\mathcal{L}t} u_0, $$
with
$$ e^{-\mathcal{L}t} = \sum_{k=0}^\infty \frac{(-t)^k}{k!} \mathcal{L}^k. $$

This is a positive definite self-adjoint operator on $L^2 [0,1]$, and we may use it to write the "heat kernel" for measurable subsets of $[0,1]$. Given $U,V \in \mathfrak{M}_0$ and $t\ge 0$, we define
$$ q_t (U,V) = \langle \II_V, e^{-\mathcal{L}t} \II_U \rangle. $$
In light of the discussion in Section \ref{sec:other-odes}, we will substitute the Laplacian operator by the adjacency operator in order to avoid unnecessary computations. Thus,
$$ p_t(U,V) = \langle \II_V, e^{\mathcal{W}t} \II_U \rangle. $$

From there, we use the Varadhan formula from Equation \eqref{eq:graph-varadhan} to define a symmetric premetric on measurable subsets. Given $U,V \in \frakM$,
$$ \delta_W (U,V) = \lim_{t\rightarrow 0^+} \frac{\log p_t (U,V)}{\log t}. $$
The adjacency operator $\calW$ is not necessarily positive semidefinite. Nonetheless, the conclusion of Theorem \ref{thm:keller} still holds, because $W$ is non-negative, and we only consider inner products with non-negative functions $\II_U, \II_V$, so $p_t (U,V)$ is always non-negative. We therefore have
\begin{equation} \label{eq:delta_W-inf} 
    \delta_W (U,V) = \inf \{m \in \N \,|\, \langle \II_U, \calW^m \II_V \rangle > 0\}.
\end{equation}

The function $\delta_W$ does not satisfy most of the axioms of a metric. First, if $\mu(U \cap V) \neq 0$, then $\delta_W (U,V) = 0$, so $\delta_W$ is degenerate. Second, the triangle inequality is not always satisfied, as the following example shows.

Consider the cycle graph $C_6$ of length $6$, with nodes labelled in cyclic order $v_1,\dots,v_6$ and let $W = \step_6 (A)$ where $A$ is the adjacency matrix of $C_6$. We denote $I_1, \dots, I_6$ the partition sets in correspondence with the vertices $v_1,\dots, v_6$. Let $X = I_1$, $Y=I_2 \cup I_3$ and $Z = I_4$. Simple calculations (or a drawing as in Figure \ref{fig:cycle}) show that $\delta_W (X,Y) = \delta_W (Y,Z) = 1$, while $\delta_W (X,Z)=3$. The reason for this fact is that by grouping vertices together, we are skipping intermediate edges, thereby artificially shortening composite distances.

\begin{figure}
    \centering
    \includegraphics[alt={A graphical illustration of a cycle graph with 6 vertices labelled in cyclic order. Two boxes labelled X and Z surround vertices 1 and 4 respectively, while a box labelled Y surrounds vertices 2 and 3.},height=0.15\textheight]{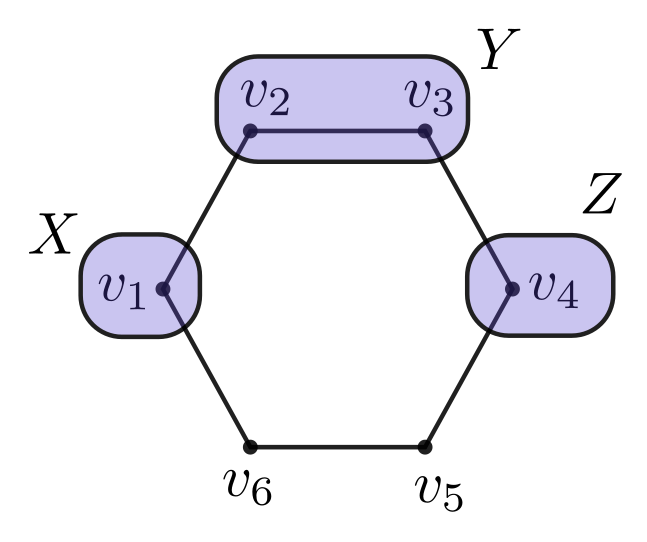}
    \caption{Schematic example of a graph with three sets $X$, $Y$, $Z$ violating the triangle inequality for $\delta_W$.\\
    Alt text: A graphical illustration of a cycle graph with 6 vertices labelled in cyclic order. Two boxes labelled X and Z surround vertices 1 and 4 respectively, while a box labelled Y surrounds vertices 2 and 3.}
    \label{fig:cycle}
\end{figure}

It turns out that all of these problems can be avoided if we coarsen $W$ with a measurable partition $P$ and restrict $\delta_{W_P}$ to the elements of $P$. Before stating the result, we will need the following elementary lemmas, whose proofs can be found in the supplementary materials.

\begin{lemma}[Action of step graphon adjacency operators] \label{lemma:step-graphon-adjacency}
    Let $A\in \AA_n$ and $P \in \calP_n$ and let $W = \step_P (A)$. Then the adjacency operator $\calW = \adj (W)$ acts on $f \in L^2 [0,1]$ as
    $$ \calW f = \sum_{i=1}^n \left( \sum_{j=1}^n A_{ij} \mu_j f_j \right) \II_{P_i}, $$
    where $\mu_i = \mu(P_i)$ and $f_i = \frac{1}{\mu_i} \int_{P_i} f(x) dx$.
\end{lemma}

\begin{lemma}[Powers of step graphons] \label{lemma:step-graphon-powers}
    Let $A \in \AA_n$, $P \in \calP_n$ and $m \ge 1$. Then
    $$ (\step_P(A))^{\circ m} = \step_P (M^{m-1}A), $$
    where $M_{ij} = A_{ij} \mu_j$.
\end{lemma}

\begin{theorem} \label{thm:delta_W-coarsen}
    Let $W\in\WW$ be a connected graphon, $P \in \calP_n$, a measurable partition, and $W_P = \coarsen_P (W)$. Then $(P,\delta_{W_P})$ is a metric space.
\end{theorem}
\begin{proof}
    We know by Proposition \ref{prop:graphon-coarsen-connected} that $W_P$ is connected. Lemmas \ref{lemma:step-graphon-adjacency} and \ref{lemma:step-graphon-powers} imply that we can interpret $A = \mat_P (W)$ as the weighted adjacency matrix of a connected graph, and we can directly apply Varadhan's formula for graphs (or its generalization described in \ref{sec:other-odes}).
\end{proof} 

\begin{corollary} \label{coro:graphon-step-isometry}
    Let $A\in \AA_n$ be the adjacency matrix of a connected finite graph $G$ and $W = \lift_n (A)$, with $I^{(n)} = \{I^{(n)}_1,\dots,I^{(n)}_n\} \in \calP_n$ the corresponding partition of $[0,1]$. Then $(I^{(n)}, \delta_W)$ is isometric to $(V(G),d_G)$, where $V(G)$ is the vertex set of $G$, and $d_G$ is the shortest-path distance on $G$.
\end{corollary}

\begin{proposition}
    Let $W \in \WW$ be connected, $P\in \calP_n$ and $\varphi : [0,1] \rightarrow [0,1]$ a measure preserving bijection. Then $(\varphi(P), \delta_{W_P^\varphi})$ is isometric to $(P,\delta_{W_P})$, where $\varphi(P) = \{\varphi(P_i) \,|\, P_i \in P\}$.
\end{proposition}
\begin{proof}
    First, notice that $\varphi(P) \in \calP_n$, so $(\varphi(P), \delta_{W_P^\varphi})$ is a metric space by Theorem \ref{thm:delta_W-coarsen}. Let $P_i, P_j \in P$, and $\calW = \adj(W_P), \calW_\varphi = \adj(W_P^\varphi)$. By Proposition \ref{prop:isomorphic-graphon-unitary}, we have
    $$ \calW_\varphi = U^\varphi \calW U^{\varphi^{-1}}, $$
    where $(U^\varphi f)(x) = f(\varphi(x))$, for $f\in L^2 [0,1]$, $x\in [0,1]$ (correspondingly for $U^{\varphi^{-1}}$).

    For $m\in \N$, we have
    \begin{align*}
        \langle \II_{\varphi(P_i)}, \calW_\varphi^m \II_{\varphi(P_j)} \rangle &= \langle \II_{\varphi(P_i)}, U^\varphi \calW^m U^{\varphi^{-1}} \II_{\varphi(P_j)} \rangle\\
        &= \langle U^{\varphi^{-1}} \II_{\varphi(P_i)}, \calW^m (U^{\varphi^{-1}} \II_{\varphi(P_j)}) \rangle = \langle \II_{P_i}, \calW^m \II_{P_j} \rangle.
    \end{align*}
    It is then straightforward from Equation \eqref{eq:delta_W-inf} to conclude that $\delta_{W_P^\varphi} (\varphi(P_i), \varphi(P_j)) = \delta_{W_P} (P_i,P_j)$.
\end{proof}

Another useful property that is satisfied by $\delta_W$ is the following. We will make use of it to prove the main result of the next section.
\begin{lemma} \label{lemma:delta-monotonic}
    If $U \subset U'$, then $\delta_W (U,V) \ge \delta_W (U',V)$, for all $V\in \frakM$.
\end{lemma}
\begin{proof}
    For $m\ge 0$, we have
    $$ \langle \II_{U}, \calW^m \II_V \rangle \le \langle \II_{U'}, \calW^m \II_V \rangle. $$
    So, if $\langle \II_U, \calW^m \II_V \rangle > 0$, then $\langle \II_{U'}, \calW^m \II_V \rangle > 0$, and by Equation \eqref{eq:delta_W-inf}, the desired inequality holds.
\end{proof}

\subsection{The Varadhan distance on graphons} \label{sec:varadhan-pointwise}

For two points $x,y \in [0,1]$, we wish to use the function $\delta_W$ defined in the previous section to obtain a pointwise distance $\overline{\delta}_W (x,y)$. Intuitively, since $\delta_W$ lets us compute a distance between disjoint sets in $\frakM$, one might hope to obtain a distance between $x$ and $y$ by looking at the distance between arbitrarily small neighbourhoods of $x$ and $y$. Formally, we define
\begin{equation} \label{eq:delta_W-pointwise}
\overline{\delta}_W (x,y) = \lim_{n\rightarrow \infty} \delta_W (X_n,Y_n),
\end{equation}
where $X_n$, $Y_n$ are respectively nested sequences of neighbourhoods of $x$ (resp. $y$), such that $\cap_{n\in\N} X_n = \{x\}$ (resp. $\cap_{n\in\N} Y_n = \{y\}$). Without loss of generality, we may assume that $X_n$ and $Y_n$ are intervals containing $x$ and $y$ respectively. It is not obvious that this yields a well-defined distance for all points, and in the coming paragraph, we will introduce an equivalent definition that makes computing distances much simpler.

Let $W$ be a connected graphon with finite diameter. We recursively define sets $(B_n)_{n\in\N}$ with $B_0 = \{(x,x) | x\in [0,1]\}$ (i.e. the "diagonal" in $[0,1]^2$), and for $n\ge 1$,
\begin{equation} \label{eq:B_n} 
    B_n = \overline{\ess} \left(\supp (W^{\circ n}) \right) \setminus \left(\bigcup_{k=0}^{n-1} B_k\right),
\end{equation}
where $\overline{\ess} (S)$ denotes the \textit{essential closure} of the set $S$, defined as the set of points $x \in [0,1]$, such that for any neighbourhood $N_x$ of $x$, $\mu (N_x \cap S) > 0$. This is a subset of the topological closure of $S$.

The pointwise distance function we propose is then defined as 
\begin{equation} \label{eq:d_W}
    d_W (x,y) = \sum_{n=0}^\infty n \II_{B_n} (x,y). 
\end{equation}    
By Equation \eqref{eq:B_n} and the assumption that $W$ is connected, the sets $B_n$ form a partition of the unit square, so $d_W$ is well-defined and $d_W (x,y) = n$ iff $(x,y) \in B_n$. In particular, $d_W (x,y) = 0$ iff $x=y$, so $d_W$ is non-degenerate.

\begin{proposition} \label{thm:varadhan-pointwise}
    For all $x,y\in [0,1]$, 
    $$\overline{\delta}_W (x,y) = d_W (x,y).$$
\end{proposition}
\begin{proof}
    Recall that $(x,y) \in \overline{\ess} \supp (W^{\circ m})$ if and only if, for all neighbourhoods $X,Y$ of $x$ and $y$, $W^{\circ m}$ is not zero almost everywhere on $X \times Y$. In other words, since $W^{\circ m}$ is non-negative,
    \begin{align*}
        \int_{X \times Y} W^{\circ m} (s,t) ds dt &= \langle \II_X, \calW^m \II_Y \rangle > 0.
    \end{align*}
    If $(x,y) \in B_m$, then $m$ is the smallest integer such that $(x,y) \in \overline{\ess} \supp (W^{\circ m})$, since by definition, $(x,y)$ is not in $\overline{\ess} \supp (W^{\circ k})$, for all $k < m$. Now, given sequences of neighbourhoods $X_n, Y_n$ of $x$ and $y$, such that $\cap_{n\in\N} X_n = \{x\}$ and $\cap_{n\in\N} Y_n = \{y\}$, we have 
    $$\delta_W (X_n,Y_n) \le m,$$ 
    and $(\delta_W (X_n,Y_n))_{n\in\N}$ is an upper-bounded increasing sequence of integers (by Lemma \ref{lemma:delta-monotonic}), so it must converge to some limit $\ell \in \N$. If $\ell < m$, then $\langle \II_X, \calW^\ell \II_Y\rangle >0$ for all neighbourhoods $X$, $Y$ of $x$ and $y$, which contradicts our hypothesis that $(x,y) \in B_m$.
\end{proof}

\begin{theorem}
    The pointwise Varadhan distance $d_W$ is a metric on $[0,1]$.
\end{theorem}
\begin{proof}
    Positivity, symmetry and non-degeneracy are immediate. It remains to show the triangle inequality.

    Let $x,y,z\in [0,1]$ and $\ell_1 = d_W (x,y)$, $\ell_2 = d_W (x,y)$. By the definition of $d_W$, there must be some arbitrarily small neighbourhoods $X, Y, Z$ of $x, y$ and $z$ respectively such that $W^{\circ \ell_1}$ is non-zero on $X\times Y$ (similarly for $W^{\circ \ell_2}$ on $Y\times Z$). Then for arbitrary $u,v \in [0,1]$, we may factor $W^{\circ \ell_1 + \ell_2}$ as
    \begin{align*}
        W^{\circ \ell_1 + \ell_2} (u,v) &= \int_{[0,1]} W^{\circ \ell_1} (u,s) W^{\circ \ell_2} (s,v) ds,\\
        &\ge \int_{Y} W^{\circ \ell_1} (u,s) W^{\circ \ell_2} (s,v) ds,
    \end{align*}
    so that $W^{\circ \ell_1 + \ell_2}$ is non-zero on $X\times Z$, hence $d_W (x,y) \le \ell_1 + \ell_2 = d_W(x,y) + d_W (y,z)$. 
 \end{proof}

\begin{proposition}
    Let $A\in\AA_n$ be the adjacency matrix of a simple connected graph $G$ and $W = \lift_n (A)$ its corresponding graphon. Then the Varadhan distance $d_W$ is given for almost all $x,y \in [0,1]$, by
    $$ d_W(x,y) = \begin{cases}
        0 & \text{ if } x=y,\\
        2 & \text{ if } (x\neq y) \wedge (x,y \in I^{(n)}_i),\\
        d_G(i,j) & \text{ if } x\in I^{(n)}_i, y \in I^{(n)}_j~ (i\neq j),
    \end{cases} $$
    where $d_G$ is the shortest-path distance on $G$.
\end{proposition}

\begin{proof}
    The case $x=y$ is trivial. Suppose that $x\neq y$, such that both $x \in I^{(n)}_i$ and $y \in I^{(n)}_j$ lie in the interior of their respective partition sets. 
    
    If $i \neq j$, it is immediate, by Corollary \ref{coro:graphon-step-isometry}, that $d_W(x,y)=d_G(i,j)$. 
    
    If $i=j$, let $k$ be an arbitrary neighbour of $i$ in $G$, and $X, Y \subset I^{(n)}_i$ be arbitrary disjoint neighbourhoods of $x$ and $y$. Clearly $\delta_W (X, I^{(n)}_k) = \delta_W (I^{(n)}_k, Y) = 1$, so $\delta_W (X,Y) \le 2$. In addition, $\langle \II_X, \calW \II_Y \rangle = 0$, so $\delta_W (X,Y) > 1$, hence $\delta_W(X,Y) = 2$. Injecting this into Equation \eqref{eq:delta_W-pointwise} then directly yields $d_W(x,y) = 2$.
\end{proof}

The remaining points on the boundaries of the partition sets form a set of measure zero. If $x$ or $y$ are in the boundaries of multiple sets, then $d_W(x,y)$ is the smallest distance between those sets. The points $(\frac{1}{2},y)$ for $y\in [0,1]$ in Example \ref{ex:varadhan-bipartite} are the prototypical example.

\begin{proposition} \label{prop:d_W-isomorphism}
    Let $W$ be a connected graphon and $\varphi : [0,1] \rightarrow [0,1]$ a measure preserving bijection. Then for almost all $x,y \in [0,1]$,
    $$ d_{W^\varphi} (x,y) = d_W (\varphi(x),\varphi(y)). $$
\end{proposition}
\begin{proof}
    By Corollary \ref{coro:powers-isomorphism}, we know that $(W^\varphi)^{\circ n} = (W^{\circ n})^\varphi$, so for almost all $(x,y) \in \supp ((W^\varphi)^{\circ n})$, 
    $$(\varphi(x), \varphi(y)) \in \supp (W^{\circ n}).$$
    By denoting $B_n$ and $B_n^\varphi$ the sets defined by Equation \eqref{eq:B_n} with $W$ and $W^\varphi$ respectively, this implies that 
     $$(\varphi(x),\varphi(y)) \in B_n,$$
    for almost all $(x,y)\in B_n^\varphi$, so $d_{W^\varphi} (x, y) = d_W(\varphi(x),\varphi(y))$.
\end{proof}

Let us comment on why Proposition \ref{prop:d_W-isomorphism} only holds for almost all $(x,y) \in [0,1]$. The application of a measurable bijection to the support of $W^{\circ n}$ can change its boundary, but only by a zero-measure set. Thus, a point in $B_n$ can find itself on the boundary of $B_m^\varphi$, with $m<n$ after going through $\varphi$.

\begin{example}
Let $W = p\II_{[0,1]^2}$ be the Erd\"os-Rényi graphon, for $p\in[0,1]$. The pointwise Varadhan distance is given by
$$ d_W (x,y) = \begin{cases}
    0 & \text{ if }x=y,\\
    1 & \text{ if } x\neq y.
\end{cases} $$ 
Note that $d_W$ does not depend on the parameter $p$. This is consistent with known results about Erd\"os-Rényi random graphs that state that the expected shortest path length between two nodes is proportional to $\frac{\ln (n)}{\ln (pn)}$ \cite{wattsCollectiveDynamicsSmallworld1998}. Taking the limit as $n$ goes to infinity of this expression indeed yields one. Finally, let us remark that any trivially connected graphon will have the same Varadhan distance.
\end{example}

\begin{example} \label{ex:varadhan-bipartite}
Let $W$ be the bipartite graphon. The Varadhan distance is given by
$$
d_W(x,y) = \begin{cases}
    0 & \text{ if } x=y,\\
    1 & \text{ if }(x,y) \in [0,\frac{1}{2}] \times [\frac{1}{2},1] \text{ or } (x,y) \in [\frac{1}{2},1] \times [0,\frac{1}{2}],\\
    2 & \text{else}.
\end{cases} $$

See Figure \ref{fig:ex-varadhan} panels A and B for an illustration.
\end{example}

\begin{example} \label{ex:varadhan-circular}
Let $W$ be the circular band graphon 
$$ W(x,y) = \begin{cases}
    1 & \text{ if } \Delta(x,y) \le \tau,\\
    0 & \text{ if } \Delta(x,y) > \tau,
\end{cases} $$
where $\Delta(x,y) = \min(|x-y|, 1-|x-y|)$ and $\tau>0$. Simple, but tedious calculations show that the composition power $W^{\circ n}$ is piecewise polynomial, and its support is a circular band of width $n\tau$. In the case of $\tau = \frac{1}{7}$, the diameter of $W$ is equal to four, and the Varadhan distance is given by
$$ d_W (x,y) = \lceil \Delta(x,y) / \tau \rceil. $$

See Figure \ref{fig:ex-varadhan} panels C and D for an illustration.
\end{example}

\begin{figure}
    \centering
    \includegraphics[alt={Graphical representations of the bipartite graphon (A) and circular band graphon (C) and of their respective Varadhan distance function, (B and D).},width=.9\textwidth]{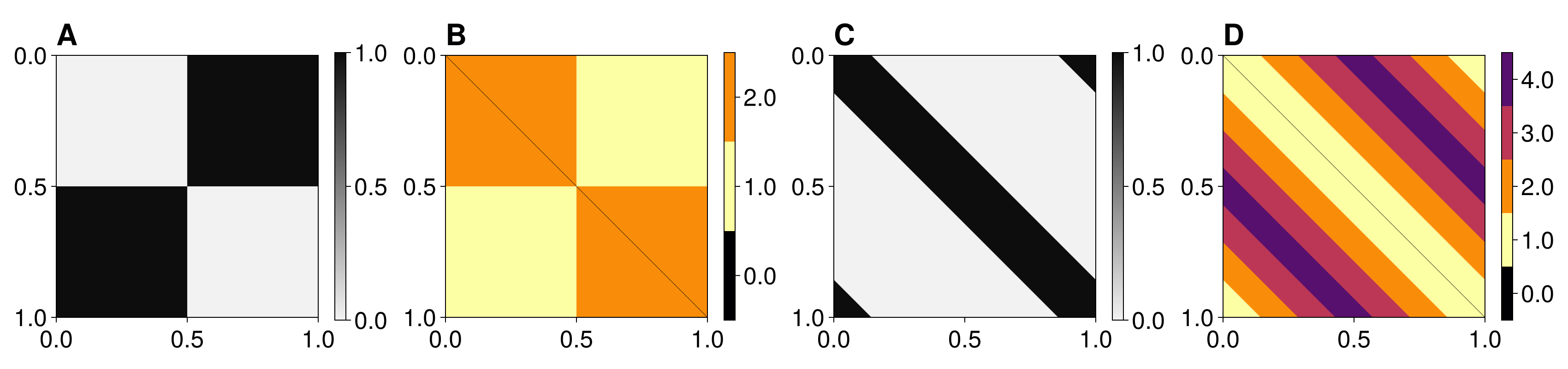}

    \caption{The Varadhan distance from examples \ref{ex:varadhan-bipartite} and \ref{ex:varadhan-circular}. (A) The bipartite graphon. (B) Varadhan distance on the bipartite graphon. (C) The circular band graphon with parameter $\tau=1/7$. (D) The Varadhan distance on the circular band graphon.\\
    Alt text: Graphical representations of the bipartite graphon and circular band graphon and of their respective Varadhan distance function.}
    \label{fig:ex-varadhan}
\end{figure}

%% file: sections/communicability.tex
% LTeX: language=en-GB
\section{The communicability distance on graphons} \label{sec:communicability}

The Varadhan distance developed in the previous sections generalizes the shortest path distance to graphons, but otherwise lacks many desirable properties. In particular, because it is integer-valued, the topology it induces on the unit interval is the discrete topology, which is too strong for many problems of interest. Moreover, as a consequence of Proposition \ref{thm:varadhan-pointwise}, the Varadhan distance only depends on the support of $W$, and its composition powers, so for instance, any trivially connected graphon will have the same pointwise distance as the Erd\"os-Rényi graphon.

In this section, we draw a connection between the Varadhan distance and the \textit{communicability distance} \cite{estradaCommunicabilityDistanceGraphs2012}. Because all composition powers of $W$ appear in its definition, this latter distance includes information from all paths, not just shortest paths, and as a result is more granular than the Varadhan distance. We will however restrict ourselves to the elementary definitions, and reserve the thorough treatment of this matter for follow-up work.

The communicability distance for a finite graph with adjacency matrix $A \in \AA_n$ is defined by
$$ d^2_C (x,y) = (e^A)_{x,x} + (e^A)_{y,y} - 2 (e^A)_{x,y}. $$
This expression is equivalent to the induced metric
$$ d_C (x,y) = \left\| e^{\frac{A}{2}} (\delta_x - \delta_y) \right\|_2, $$
where $\delta_x, \delta_y \in \R^n$ denote canonical basis vectors as in Section \ref{sec:varadhan-graphs}. Furthermore, the communicability distance comes with a natural isometric embedding of the nodes into a Euclidean space, given by $x \mapsto e^{\Lambda/2} v_x$, where $\Lambda \in \R^{n\times n}$ is the diagonal matrix of eigenvalues of $A$, with corresponding eigenvector matrix $V$ such that $A = V\Lambda V^T$ and $v_x = [V_{x,1}, \dots, V_{x,n}]^T \in \R^n$.

All those facts may be straightforwardly extended to graphons. First, for a graphon $W$ with adjacency operator $\calW$, we define the communicability distance between two sets $X,Y \in \frakM$ by
\begin{equation} \label{eq:communicability}
    d_{C,W} (X,Y) = \left\| e^{\frac{\calW}{2}} (\II_X - \II_Y) \right\|_2. 
\end{equation} 
Unlike the quantity $\delta_W$ defined in the previous sections, this is a well-defined metric on $\frakM$. This is because the exponential $e^{\frac{\calW}{2}}$ is invertible, so $d_{C,W}$ is essentially the metric induced by $e^{\frac{\calW}{2}}$ on $L^2 [0,1]$. Let us remark that as a technicality, we need to identify the sets $X,Y \in \frakM$ that have the same essential support, i.e. $\II_X = \II_Y$ almost everywhere. The resulting metric space is not complete, because one may take sets of arbitrarily small measure to construct a sequence of sets $(X_n)_{n\in\N} \subset \frakM$ such that $(\II_{X_n})_{n\in\N}$ converges to the zero function. Adding back the zero measure sets makes the space complete.

It is also possible to extend the communicability distance's isometric embedding to the graphon case. This is because the adjacency operator $\calW$ is compact and self-adjoint, and thus admits a spectral decomposition of the form
\begin{equation} \label{eq:adjacency-eigen} 
    \calW f = \sum_{k=0}^\infty \lambda_k \langle \varphi_k, f \rangle \varphi_k,
\end{equation}
for $f\in L^2 [0,1]$, where $(\lambda_k)_{k\in\N} \subset \R$ are the eigenvalues of $\calW$ and $(\varphi_k)_{k\in\N}$ is the corresponding basis of orthonormal eigenfunctions (which can moreover be shown to be bounded \cite{lovaszGraphon2012}). The exponential can then be expressed as
\begin{equation} \label{eq:communicability-eigen} 
    e^{\frac{\calW}{2}} f = \sum_{k=0}^\infty e^{\frac{\lambda_k}{2}} \langle \varphi_k, f \rangle \varphi_k.
\end{equation}

It is then clear that the map
$$ X \mapsto (e^{\frac{\lambda_k}{2}} \langle \II_X, \varphi_k \rangle )_{k\in\N}, $$
from $\frakM$ to $\ell^2$ is the analogue of the communicability embedding from finite graphs, and is isometric with respect to the communicability distance $d_{C,W}$.

Before concluding, let us comment several points and related work. We could have defined a family of distances similar to \eqref{eq:communicability} by using the heat semigroup $e^{-\calL t}$ instead. Unlike the communicability distance, these would not have isometric embeddings, as the Laplacian operator is not compact and therefore typically does not admit a spectral decomposition as in \eqref{eq:adjacency-eigen}. One way to circumvent this is to use the \textit{normalized Laplacian} $\overline{\calL}$, defined by 
\begin{equation} \label{eq:normalized-laplacian} 
    (\overline{\calL} f)(x) = f(x) - \int_0^1 \frac{W(x,y)}{\sqrt{k(x) k(y)}} f(y) dy
\end{equation}
instead, which would then yield an embedding equivalent to the Diffusion Maps algorithm \cite{coifmanDiffusionMaps2006}. The normalized Laplacian comes with other issues, however, as it is not always well-defined, even for connected graphons. The reason for this is that the degree function of a connected graphon can be equal to zero on a null set, which would make equation \eqref{eq:normalized-laplacian} ill-defined.

%% file: sections/conclusion.tex
% LTeX: language=en-GB
\section{Conclusion}

In this work, we have extended the notion of shortest-path distance from finite graphs to graphons by using an analytic characterization, and further highlighted the link between the shortest-path distance and the communicability distance. More precisely, we have constructed the distance function $\delta_W$ by using the exponential operators $e^{\calW t}$, which yields a distance when restricted to measurable partitions, and a pointwise distance on the unit interval when taking appropriate limits. The same exponential operator enabled us to define the communicability distance as a metric on measurable subsets.

Much like the work in \cite{kellerNoteShorttimeBehavior2016}, which proved their result as a consequence of a general theorem about self-adjoint operators, most of our results can be extended to a more general operator setting, which should facilitate extending them to other types of graph limits \cite{backhauszActionConvergenceOperators2022,borgsLpTheorySparse2019}.

There are however some important aspects of graphon theory which we have not addressed here. First, the interpretation of graphons as random graph models is one of their most important applications, and for which characterizing the shortest-path length distribution is of interest \cite{loombaGeodesicLengthDistribution2025}. Second, we have not discussed convergence results of either the Varadhan or communicability distance, for a sequence of graphons converging in the cut-norm or some other topology. Third, there are many other distances on graphs defined in terms of matrices such as the resistance distance \cite{kleinResistanceDistance1993}, forest distance \cite{chebotarevForestMetricsGraph2002} and others \cite{chebotarevClassGraphgeodeticDistances2011,chebotarevWalkDistancesGraphs2012}. We expect most of these to be amenable to being extended to graphons.

We plan to address those points in future work, along with the following. There are many results in spectral graph theory relating the spectrum of a graph to its shortest-path lengths (see for example, chapters 2-3 of \cite{chungSpectralGraphTheory1997} that contain many of the ideas used in our construction of the function $\delta_W$), some of which have been extended to graphons \cite{khetanCheegerInequalitiesGraph2024}. It seems worthwhile to investigate the relation between these results and ours.

\section*{Acknowledgements}

We thank J. Winkin and A. Hastir for their helpful conversations at the onset of this project, as well as F. Garin for stimulating comments in private communications.

%% file: appendix.tex
% LTeX: language="en-GB"
\appendix

\section{A general Varadhan's formula for graphs} \label{sec:general-varadhan}

In this section, we prove a more general form of Varadhan's formula for graphs. We specifically consider matrices of the form
$$ L = M + D, $$
where $M$ is a matrix with positive entries such that $M_{ij}=0$ iff $A_{ij}=0$ and $D$ is an arbitrary diagonal matrix. This encompasses the adjacency matrix and the various graph Laplacians. 

More precisely, we will prove the following proposition.

\begin{proposition} \label{prop:varadhan-graph-general}
    Let $A\in\AA_n$ be the adjacency matrix of a connected graph and $M \in \R_+^{n\times n}$ a matrix with positive entries such that $M_{ij}=0$ iff $A_{ij}=0$ for all $i,j \in \{1,\dots,n\}$, $D \in \R^{n\times n}$ a diagonal matrix, and $f : \R \rightarrow \R$ a real analytic function around zero such that for $t$ near zero,
    $$ f(t) = \sum_{k=0}^\infty \alpha_k t^k, $$
    with $\alpha_k \neq 0$, for all $k\in \N$. Then 
    \begin{equation} \label{eq:varadhan-general} 
        \lim_{t\rightarrow 0^+} \frac{\log |f(L t)_{ij}|}{\log t} = d(i,j), 
    \end{equation}
    where $d(i,j)$ is the shortest path on the graph corresponding to $A$, $L=M+D$ and
    $$ f(L t) = \sum_{k=0}^\infty \alpha_k t^k L^k. $$
\end{proposition}

This is not the most general result possible. For example, we could flip the sign of $M$, or use the adjacency matrix of a strongly connected directed graph. It will not work for signed adjacency matrices (with positive and negative edges).

We will first need three simple lemmas. The proofs are elementary, and we provide them for completeness. 

\begin{lemma} \label{lemma:log-taylor}
    Let $f(t) = \sum_{k=0}^\infty \alpha_k t^k$ be a real analytic function around zero.
    \begin{enumerate}
        \item If $\alpha_0 \neq 0$, then
        $$ \lim_{t\rightarrow 0^+} \frac{\log |f(t)|}{\log t} = 0. $$
        \item If $\ell \in \N$ is such that $\alpha_\ell \neq 0$ and $\alpha_k = 0, \forall k < \ell$, then
        $$ \lim_{t\rightarrow 0^+} \frac{\log |f(t)|}{\log t} = \ell. $$
    \end{enumerate} 
\end{lemma}
\begin{proof}
    In the first case, it is obvious that $f(0)=\alpha_0$ and $|f(t)|$ is analytic in a neighborhood of zero, so $\lim_{t\rightarrow 0^+} |f(t)| = |\alpha_0| > 0$, and since the logarithm function is continuous on $(0,+\infty)$, $\log |f(t)|$ converges to $\log |\alpha_0|$ as $t$ goes to zero. It is then immediate that
    $$ \lim_{t \rightarrow 0^+} \frac{\log |f(t)|}{\log t} = 0. $$
    
    In the second case, we may factor $f(t)$ as
    $$ f(t) = t^\ell \left( \sum_{k=0}^\infty \alpha_{\ell+k} t^k \right), $$
    hence for $t>0$, $\log |f(t)| = \ell \log t + \log |\sum_{k=0}^\infty \alpha_{\ell+k} t^k|$ and so,
    \begin{align*}
        \lim_{t\rightarrow 0^+} \frac{\log |f(t)|}{\log t} &= \lim_{t\rightarrow 0^+} \frac{\ell \log t}{\log t} + \lim_{t \rightarrow 0^+} \frac{\log |\sum_{k=0}^\infty \alpha_{\ell+k} t^k|}{\log t},\\
        &= \ell,
    \end{align*}
    where the second limit is zero by the first point of this lemma.
\end{proof}

\begin{lemma} \label{lemma:weighted-adjacency-powers}
    Let $A \in \AA_n$ be the adjacency of a connected simple graph and let $M \in \R_+^{n\times n}$ be such that $M_{ij}=0$ if and only if $A_{ij}=0$ for all $i,j\in \{1,\dots,n\}$. Then,
    $$ (M^m)_{ij} = 0 \; \Leftrightarrow \; (A^m)_{ij} = 0, $$
    for all $m\in \N$, and all $i,j \in \{1,\dots,n\}$.
\end{lemma}
\begin{proof}
    The case $m=0$ is trivial, and we have the case $m=1$ by hypothesis. We proceed by induction.

    Let $m \in \N$ be such that $(M^m)_{ij} = 0$ iff $(A^m)_{ij}=0$. Then, we have
    \begin{align*}
        (M^{m+1})_{ij} &= \sum_{k=1}^n (M^m)_{ik} M_{kj},\\
        (A^{m+1})_{ij} &= \sum_{k=1}^n (A^m)_{ik} A_{kj}.
    \end{align*}
    If $(M^{m+1})_{ij} \neq 0$, there must be some $k \in \{1,\dots,n\}$ such that $(M^m)_{ik} \neq 0$ and $M_{kj} \neq 0$. By the hypotheses, this is equivalent to $(A^m)_{ik}\neq 0$ and $A_{kj}\neq 0$, so $A^{m+1}_{ij}\neq 0$.
\end{proof}

Since the matrices $M$ and $D$ in proposition \ref{prop:varadhan-graph-general} do not necessarily commute, we will need to handle it explicitly when manipulating powers of $L$. For $m\ge 1$, let $\mathcal{B}_m = \{0,1\}^m$ be the set of all binary $m$-tuples. For $b \in \mathcal{B}_m$ and two matrices $A,B \in \R^{n\times n}$, we define
$$ \pi_b (A,B) = \prod_{k=1}^m A^{1-b_k} B^{b_k}, $$
to be the product of $m$ terms taken from $\{A,B\}$ indexed by $b$, e.g. $\pi_{(1,0,0)}(A,B) = BA^2$.

\begin{lemma} \label{lemma:noncommutative-binomial-formula}
    Let $A,B \in \R^{n\times n}$ be arbitrary matrices. Then for all $m\ge 1$, 
    $$ (A+B)^m = \sum_{b\in \mathcal{B}_m} \pi_b (A,B). $$
\end{lemma}
\begin{proof}
    The case $m=1$ is immediate, and we proceed by induction. Suppose that
    $$ (A+B)^m = \sum_{b\in \mathcal{B}_m} \pi_b (A,B), $$
    for $m\ge 1$, then
    \begin{align*}
        (A+B)^{m+1} &= (A+B) (A+B)^m,\\
        &= A \sum_{b\in \mathcal{B}_m} \pi_b (A,B) + B \sum_{b\in \mathcal{B}_m} \pi_b (A,B).
    \end{align*}
    Since any element of $\mathcal{B}_{m+1}$ can be expressed as either $(0, b_1,\dots,b_m)$ or $(1,b_1,\dots,b_m)$, where $(b_1,\dots,b_m)\in \mathcal{B}_m$, we see that all elements of $\mathcal{B}_{m+1}$ are accounted for in the last expression.
\end{proof}

We may finally prove proposition \ref{prop:varadhan-graph-general}.

\begin{proof}
    We first need to show that $f(L t)$ is well-defined in a neighbourhood of zero. We can expand $f(Lt)_{ij}$ into
    $$ f(Lt)_{ij} = \sum_{k=0}^\infty \alpha_k t^k (L^k)_{ij}. $$
    Let $K = \max_{i,j} |L_{ij}|$. We can bound $L$ entrywise by the $n\times n$ matrix with all entries equal to $K$ and consequently,
    $$ |(L^k)_{ij}| \le K^k n^{k-1}, $$
    where the term $n^{k-1}$ corresponds to the number of walks of length $k$ in the complete graph with self-loops with $n$ vertices. Thus,
    $$ \sum_{k=0}^\infty |\alpha_k t^k (L^k)_{ij}| \le \frac{1}{n} \sum_{k=0}^\infty |\alpha_k| (Knt)^k = \frac{1}{n} f(Knt). $$
    Since $K$ and $n$ are constants, there exists a neighbourhood $U$ of zero such that $f(Knt)$ is well-defined for all $t\in U$. The series $f(Lt)_{ij}$ is then absolutely convergent, and therefore converges for all $t\in U$.

    Next, by using lemmas \ref{lemma:noncommutative-binomial-formula} and \ref{lemma:weighted-adjacency-powers}, we know that $(L^m)_{ij} = 0$ for $m<d(i,j)$, and $(L^{d(i,j)})_{ij} = (M^{d(i,j)})_{ij} > 0$. This means that the function $t \mapsto f(Lt)_{ij}$ satisfies 
    $$ f(Lt)_{ij} = \sum_{k=d(i,j)}^\infty \alpha_k t^k (L^k)_{ij}, $$
    and we can apply lemma \ref{lemma:log-taylor} to conclude the proof.
\end{proof}

\section{Properties of step graphons} \label{sec:step-graphons}

This appendix states and proves several elementary properties of step graphons which are used in the proof of Theorem 4.1 in the main text.

\begin{lemma}[Action of step graphon adjacency operators] \label{lemma:step-graphon-adjacency-app}
    Let $A\in \AA_n$ and $P \in \calP_n$ and let $W = \step_P (A)$. Then the adjacency operator $\calW = \adj (W)$ acts on $f \in L^2 [0,1]$ as
    $$ \calW f = \sum_{i=1}^n \left( \sum_{j=1}^n A_{ij} \mu_j f_j \right) \II_{P_i}, $$
    where $\mu_i = \mu(P_i)$ and $f_i = \frac{1}{\mu_i} \int_{P_i} f(x) dx$.
\end{lemma}
\begin{proof}
    Let $x \in P_i \subset [0,1]$. By definition of the adjacency operator, we have
    \begin{align*}
        (\calW f)(x) &= \int_0^1 W(x,y) f(y) dy = \sum_{j=1}^n \int_{P_j} W(x,y) f(y) dy\\
        & = \sum_{j=1}^n A_{ij} \int_{P_j} f(y) dy = \sum_{j=1}^n A_{ij} \mu_j f_j.
    \end{align*}
    This is valid for any $x \in P_i$, so we have just shown that 
    $$ (x \in P_i) \Rightarrow (\calW f)(x) = \sum_{j=1}^n A_{ij} \mu_j f_j, $$
    for all $i \in \{1,\dots,n\}$, hence we obtain the desired result.
\end{proof}

From the above lemma, we see that $\calW$ acts on the finite dimensional subspace $\mathcal{S}_P$ of $P$-step functions as the matrix $(M_{ij}) = (A_{ij} \mu_j)$. Moreover, $\mathcal{S}_P$ is the range of $\calW$.

Applying this lemma to the constant function $\II_{[0,1]}$ immediately gives the following result.
\begin{lemma}[Degree function of step graphons]
    Let $A\in \AA_n$, $P \in \calP_n$ and $W = \step_P (A)$. The degree function of $W$ is given by
    $$ k = \sum_{i=1}^n k_i \II_{P_i}, $$
    where $k_i = \sum_{j=1}^n A_{ij} \mu_j$.
\end{lemma}

\begin{lemma}[Powers of step graphons] \label{lemma:step-graphon-powers-app}
    Let $A \in \AA_n$, $P \in \calP_n$ and $m \ge 1$. Then
    $$ (\step_P(A))^{\circ m} = \step_P (M^{m-1}A), $$
    where $M_{ij} = A_{ij} \mu_j$.
\end{lemma}
\begin{proof}
    Let $W = \step_P (A)$ and $\calW = \adj (W)$. By lemma \ref{lemma:step-graphon-adjacency-app}, we know that $\operatorname{range}(\calW) = \mathcal{S}_P$ and that for $f = \sum_{i=1}^n f_i \II_{P_i} \in \mathcal{S}_P$,
    $$ \calW f = \sum_{i=1} (M [f])_i \II_{P_i}, $$
    where $[f] = [f_1,\dots,f_n]^T \in \R^n$. From there it is immediate that
    $$ \calW^m f = \sum_{i=1}^n (M^m [f])_i \II_{P_i}. $$

    Now let $Z = \step_P (M^{m-1} A)$ and $\mathcal{Z} = \adj(Z)$. By lemma \ref{lemma:step-graphon-adjacency-app}, we know that
    \begin{align*}
        \mathcal{Z} f &= \sum_{i=1}^n \left( \sum_{j=1}^n (M^{m-1} A)_{ij} \mu_j f_j \right) \II_{P_i}\\
        &= \sum_{i=1}^n (M^m [f])_i \II_{P_i} = \calW^m f,
    \end{align*}
    from which we conclude that $W^{\circ m} = \step_P (M^{m-1}A)$ almost everywhere.
\end{proof}

\begin{corollary} \label{coro:powers-step-homogeneous}
    Let $A \in \AA_n$ and $P \in \calP_n$ be a homogeneous partition. Then
    $$ (\step_P(A))^{\circ m} = \step_P \left(\frac{1}{n^{m-1}}A^m \right), $$
\end{corollary}
\begin{proof}
    Homogeneity of the partition implies $\mu(P_i) = \frac{1}{n}$ for all $i \in \{1,\dots,n\}$, hence the matrix $M$ is simply $M = \frac{1}{n}A$, and applying lemma \ref{lemma:step-graphon-powers-app} immediately gives the desired result.
\end{proof}